\documentclass[final,nomarks]{dmtcs-episciences}

\usepackage[utf8]{inputenc}
\usepackage{subfigure}

\usepackage{amsmath, amssymb, amsfonts, amsthm}

\DeclareMathOperator{\Av}{Av}
\renewcommand{\S}{\mathcal{S}}

\newcommand{\I}{\mathcal{I}}

\DeclareMathOperator{\inv}{inv}
\DeclareMathOperator{\cyc}{cyc}
\DeclareMathOperator{\fix}{fp}

\DeclareMathOperator{\exc}{exc}

\usepackage{bm}

\theoremstyle{plain}
\newtheorem{theorem}{Theorem}
\newtheorem{lemma}[theorem]{Lemma}

\newtheorem{corollary}[theorem]{Corollary}

\theoremstyle{definition}

\author{Kassie Archer}
\title{Enumerating two permutation classes by the number of cycles}
\affiliation{University of Texas at Tyler, USA}
\received{2020-02-28}
\revised{2022-09-01} 
\accepted{2022-09-08}

\begin{document}

\publicationdetails{22}{2022}{2}{11}{6173}

\maketitle

\begin{abstract}
 We enumerate permutations in the two permutation classes 
 $\Av_n(312, 4321)$ and $\Av_n(321, 4123)$ by the number of cycles each permutation admits. 
 
 We also refine this enumeration with respect to several statistics.
\end{abstract}

\section{Introduction}
 
 Pattern avoidance is inherently a property of a permutation's one-line notation. For this reason, it can be difficult at times to enumerate pattern-avoiding permutations with respect to their cycle type or number of cycles. For example, it remains unknown how many cyclic permutations avoid a single pattern of length 3. 
 
  A few results on cycle type or number of cycles have been found for permutations avoiding several patterns. Among these are unimodal permutations, i.e. those avoiding 312 and 213 \cite{THIBON2001, GANNON2001}; 
 other pairs of length 3 permutations \cite{ELIZALDEPHD04, BonaCory}; 
 almost-increasing permutations, which avoid four length 4 permutations \cite{ELIZALDE11B, KL2017}; 
 and permutations from certain grid classes \cite{GR93, DFH2013,AE2014, AL2016}. 
In addition the numbers of cyclic permutations of length $n$ avoiding 123 or 321 as a \textit{consecutive} pattern are given in \cite{ET2018}.
 There are also many results regarding pattern-avoiding involutions and fixed points of pattern-avoiding permutations \cite{RSZ2002,EI04, DRS2007, BBS2011, ELIZALDE2012}. 

 In this paper, we consider two permutation classes, namely the set of permutations avoiding 312 and 4321 and the set of permutations avoiding 321 and 4123. The enumeration of each of these pattern classes by size of the permutation can be found in \cite{Atkinson99,West96}; they are both enumerated by the Fibonacci numbers of even index, $F_{2n}$. The structure of these permutations make them amenable to enumeration by number of cycles. In Section~\ref{312 4321}, we enumerate permutations avoiding 312 and 4321 by the number of cycles and then refine this enumeration with respect to excedances and number of inversions. Finally, we enumerate the involutions avoiding those two patterns with respect to number of cycles, fixed points, and excedances. In Section~\ref{321 4123}, we do the same for permutations avoiding 321 and 4123. 
 
 
\subsection{Permutations}
 Let $\S_n$ denote the set of permutations of the set $[n]=\{1,2,\ldots, n\}$. We write a permutation $\pi\in \S_n$ in its one-line notation as $\pi = \pi_1\pi_2\ldots \pi_n$ where $\pi_i = \pi(i)$. Alternatively, we can write the same permutation in its standard cycle notation as a product of disjoint cycles. For example, the permutation $\pi = 32584167$ written here in its one-line notation can also be written in its cycle notation as $\pi = (1354876)(2)$. 
  
  \subsection{Pattern avoidance}
 Given a permutation $\pi\in\S_n$ and a permutation, or \textit{pattern}, $\sigma\in\S_k$, we say that $\pi$ \emph{contains} $\sigma$ if there is a set of indices $i_1<i_2<\cdots<i_k$ so that $\pi_{i_1}\pi_{i_2}\ldots\pi_{i_k}$ is in the same relative order as $\sigma$. If $\pi$ does not contain $\sigma$, we say that $\pi$ \emph{avoids} $\sigma$. For example, the permutation $\pi = 31562487$ contains $\sigma=123$ (for example, 356 and 368 are occurrences) and $312$ (for example, 312 and 524 are occurrences), but it avoids the pattern $\sigma =321$. 
 
 We denote the set of permutations that avoid a pattern $\sigma$ by $\Av(\sigma)$ and the set of permutations in $\S_n\cap \Av(\sigma)$ by either $\Av_n(\sigma)$ or $\S_n(\sigma)$. We say a permutation $\pi$ avoids a set of patterns $\sigma_1,\sigma_2, \ldots, \sigma_r$ if it avoids each $\sigma_i$ for $1\leq i\leq r$. In this case, we denote the set of permutations that avoid these patterns by $\Av(\sigma_1, \sigma_2,\ldots, \sigma_r)$ and the set of permutations in $\S_n\cap\Av(\sigma_1, \sigma_2,\ldots,\sigma_r)$ by either $\Av_n(\sigma_1, \sigma_2,\ldots,\sigma_r)$ or $\S_n(\sigma_1, \sigma_2,\ldots,\sigma_r)$. We call a set of permutations closed under pattern containment a \emph{permutation class}. In particular any set of permutations that can be characterized as avoiding a set of permutations is a permutation class.

  \subsection{Permutation statistics} \label{statistics}
  Given a permutation $\pi \in \S_n$, we will denote the number of cycles in the cycle notation of $\pi$ by $\cyc(\pi)$. We say that $i$ is a \emph{fixed point} of $\pi$ if $\pi_i=i$. We denote the number of fixed points of $\pi$ by $\fix(\pi)$. An element $j$ is an \emph{excedance} if $\pi_j>j$ and we say it is a nonexcedance otherwise. The number of excedances is denoted by $\exc(\pi)$. We say that a pair $(i,j)$ is an \emph{inversion} of $\pi$ if $i<j$ and $\pi_i>\pi_j$. 
  For example, if $\pi=31642875 = (136852)(4)(7)$, then $\cyc(\pi) = 3$, $\fix(\pi) = 2$, $\exc(\pi)= 3$, and $\inv(\pi) = 9$.

  We say $\pi$ is an \emph{involution} if it is its own algebraic inverse, i.e. if $\pi^{-1}=\pi$. This is equivalent to $\pi$ being comprised of length-2 cycles and fixed points only. For example, the permutation $\pi = 1574263=(1)(25)(37)(4)(6)$ is an involution. Here, we will denote the set of involutions in $\S_n$ by $\I_n$ and the set of involutions that avoid the patterns $\sigma_1, \sigma_2, \ldots, \sigma_r$ by $\I_n(\sigma_1,\sigma_2,\ldots,\sigma_r)$.

\section{$\S_n(312, 4321)$}\label{312 4321}

In this section, we enumerate permutations that avoid $312$ and $4321$ with respect to several statistics. In each case, we find a recurrence and obtain a generating function as a result.

\subsection{Number of cycles in $\S_n(312, 4321)$}
The set $\Av(312, 4321)$ was found in \cite{West96}  to be enumerated by $F_{2n},$ the Fibonacci numbers with even index. Our first result of this section is Theorem \ref{thm 312 4321} below, where we refine this enumeration with respect to cycles.

 \begin{theorem}\label{thm 312 4321}
 Let $$A(t,z) = \sum_{n,k\geq 1} |\{\pi \in \S_n(312,4321) : \cyc(\pi) = k\}| t^kz^n.$$
 Then $$A(t,z) = \frac{tz(1-z^2)}{1-(1+t)(z+z^2)+tz^3}.$$
 \end{theorem}
 
The first few terms of this generating function are as follows:
\begin{align*}
A(t,z) &= tz + (t+t^2)z^2 + (t+3t^2 + t^3)z^3 + (2t+5t^2+5t^3+t^4)z^4\\
&+ (3t+10t^2 + 13t^3 + 7t^4+ t^5) z^5 + (5t+19t^2 + 30t^3 + 25 t^4 + 9t^5 + t^6)z^6 + \cdots.
\end{align*}

As a corollary of this theorem, we enumerate cyclic permutations that avoid 312 and 4321. This follows from Theorem~\ref{thm 312 4321} by extracting the coefficient of $t$, but will also follow easily from the proof of Theorem~\ref{thm 312 4321} below.

  \begin{corollary}
  The number of \textit{cyclic} permutations avoiding the patterns 312 and 4321 is given by the Fibonacci number $F_n$ for $n\geq 2$.
  \end{corollary}

To prove Theorem~\ref{thm 312 4321}, let us first consider the following lemma. 

\begin{lemma}\label{lem where n is}
For $n\geq 1$, if $\pi \in \S_n(312, 4321)$, then the entry $n$ must appear in one of the last three entries of $\pi$. 
\end{lemma}
\begin{proof}
Suppose $\pi \in \S_n(312, 4321)$. Since $\pi$ avoids 312, any entries that appear after $n$ must be in decreasing order, and since $\pi$ avoids 4321, there can be at most two such entries. 
\end{proof}

Since for any permutation in $\pi\in \S_n(312,4321)$, we need only consider the three cases where $\pi_n=n$, $\pi_{n-1}=n$, and $\pi_{n-2}=n$. We will consider these three cases separately. 
Let $a_n(k)$ denote the number of permutations in $\S_n(312, 4321)$ that are composed of $k$ cycles. In the language of Theorem~\ref{thm 312 4321}, we have $A(t,z) = \sum a_n(k) t^kz^n$. 

\begin{lemma}\label{lem 312 n}
Let $n\geq 1$. The number of permutations $\pi \in \S_n(312, 4321)$ composed of $k$ cycles and with $\pi_n=n$ is given by $a_{n-1}(k-1)$. 
\end{lemma}
\begin{proof}
It is clear that by removing the term $\pi_n=n$, we are left with a permutation in $\S_{n-1}(312, 4321)$. In this case, since $n$ was a fixed point of $\pi$, by removing it, we have removed one cycle. This process is invertible since adding an $n$ at the end of a permutation cannot introduce a 312 or 4321 pattern.
\end{proof}

\begin{lemma}\label{lem 312 n-1}
Let $n\geq 1$. The number of permutations $\pi \in \S_n(312, 4321)$ composed of $k$ cycles and with $\pi_{n-1}=n$ is given by $a_{n-1}(k)$. 
\end{lemma}
\begin{proof}
Suppose $\pi\in\S_{n}(312, 4321)$ with $\pi_{n-1}=n$. Notice that $n-1, n,$ and $\pi_n$ appear consecutively in a cycle together. (It could be the case that $\pi_n=n-1$, in which case $n-1$ and $n$ form a 2-cycle.)
Consider the permutation $\pi'$ obtained by removing the value $n$. In this case, $\pi'\in\S_{n-1}(312,4321)$ and we now have that $\pi'_{n-1}=\pi_n$, so $n-1$ maps directly to $\pi_n$ and no new cycle has been deleted or added. 

Given any permutation $\pi'\in\S_{n-1}(312,4321)$, we can insert $n$ directly after $n-1$ in the cycle notation of $\pi'$ and obtain a permutation $\pi$ with $\pi_{n-1}=n$. This process does not change the number of cycles and cannot introduce a 312 or 4321 pattern.
\end{proof}

For example, if $\pi=345268791 = (135689)(24)(7)$, we could remove $n=9$ from the one-line notation and we have $\pi'=34526871 = (13568)(24)(7)$. The number of cycles stays the same since we also end up removing $9$ from its cycle.  

We now need to consider the number of permutations $\pi\in\S_n(312,4321)$ with $\pi_{n-2}=n$. There are two subcases to consider as illustrated in the following lemma. 

\begin{lemma}\label{lem  where n-1}
For $n\geq 1$, if $\pi \in \S_n(312, 4321)$ with $\pi_{n-2}=n$, then either $\pi_{n-1}=n-1$ or $\pi_{n-3}=n-1$.
\end{lemma}
\begin{proof}
Suppose $\pi_{n-1}\neq n-1$. We must have that $\pi_{n-1}>\pi_n$ since $\pi$ avoids 312, so we cannot have $\pi_n=n-1$. If $\pi_i=n-1$ for $i<n-3$, then either $\pi_i \pi_{n-3}\pi_{n-1}$ is an occurrence of 312 (when $\pi_{n-3}<\pi_{n-1}$) or $\pi_i\pi_{n-3}\pi_{n-1}\pi_n$ is an occurrence of 4321 (when $\pi_{n-3}>\pi_{n-1}$). Therefore, if $\pi_{n-1}\neq n-1$, we have to have $\pi_{n-3}=n-1$. 
\end{proof}

\begin{lemma}\label{lem  pt3}
Let $n\geq1$. The number of permutations $\pi \in \S_n(312, 4321)$ composed of $k$ cycles and with $\pi_{n-2}=n$ and $\pi_{n-1}=n-1$ is $a_{n-2}(k-1)$. 
\end{lemma}
\begin{proof}
Suppose $\pi \in \S_n(312, 4321)$ with $\pi_{n-2}=n$ and $\pi_{n-1}=n-1$. Then $n-2, n,$ and $\pi_n$ appear consecutively in a cycle together and $n-1$ is a fixed point. Let $\pi'$ be the permutation obtained by deleting $n$ and $n-1$. Then, $\pi' \in \S_{n-2}(312, 4321)$ with $\pi'_{n-2}=\pi_n$. In cycle notation, we have taken $\pi$, removed $n$ from its cycle so that $n-2$ maps directly to $\pi_n$, and we have removed the fixed point $n-1$. Therefore, we have decreased the number of cycles by 1. Starting with any permutation $\pi' \in \S_{n-2}(312, 4321)$, we can invert this since $n(n-1)\pi'_{n-2}$ cannot be a 312 pattern.
\end{proof}

For example, if $\pi = 245631987 = (1246)(35)(79)(8)$, then we can remove $n=9$ and $n-1=8$ from the one-line notation to get $\pi'=2456317 = (1246)(35)(7)$. Since we have removed $9$ from its cycle and have removed the fixed point $8$, we are left with exactly one fewer cycle. 

\begin{lemma}\label{lem  pt4}
Let $n\geq1$. The number of permutations $\pi \in \S_n(312, 4321)$ composed of $k$ cycles and with $\pi_{n-2}=n$ and $\pi_{n-3}=n-1$ is $a_{n-2}(k)-a_{n-3}(k-1)$. 
\end{lemma}
\begin{proof}
For a permutation $\pi \in \S_n(312, 4321)$ with $\pi_{n-2}=n$ and $\pi_{n-3}=n-1$ we must have that $n-2, n,$ and $\pi_n$  appear consecutively in a cycle together and $n-3, n-1,$ and $\pi_{n-1}$ appear consecutively in a cycle together. Removing $n$ and $n-1$ to get a permutation $\pi'\in\S_{n-2}(312, 4321)$ with the same number of cycles. 
However, we can only obtain permutations $\pi'\in\S_{n-2}(312, 4321)$ so that $\pi'_{n-3}>\pi'_{n-2}$ since $\pi$ avoided 312. Since $\pi'$ itself avoids 312 and 4321, it is the case that $\pi'_{n-3}>\pi'_{n-2}$ exactly when $\pi'_{n-2}\neq n-2$ (that is $n-2$ is not a fixed point of $\pi'$). Therefore, we obtain all permutations in $\S_{n-2}(312, 4321)$ with the same number of cycles except those where $n-2$ is a fixed point.

If a permutation in  $\S_{n-2}(312, 4321)$ has $k$ cycles and $n-2$ as a fixed point, we can remove that fixed point and get a permutation in $\S_{n-3}(312, 4321)$ with one fewer cycle. The result follows. 
\end{proof}

As an example, consider the permutation $\pi = 324168975= (134)(2)(56879)$. Removing $n$ and $n-1$, we get $\pi'= 3241675 = (134)(2)(567)$, which has the same number of cycles.

\begin{proof}[Proof of Theorem \ref{thm 312 4321}]

Since the theorem is equivalent to  $$ A = zA + ztA + z^2tA + z^2A-z^3tA+tz-tz^3,$$ is is enough to show that for 
$n\geq 4$, we have the recurrence $$a_n(k) = a_{n-1}(k) + a_{n-1}(k-1) + a_{n-2}(k-1) + a_{n-2}(k)-a_{n-3}(k-1)$$ since this recurrence together with the initial conditions gives us the functional equation.
By Lemmas~\ref{lem where n is} and \ref{lem where n-1}, the set $\S_{n}(312,4321)$ can be partitioned in four subsets. By Lemmas~\ref{lem 312 n}, \ref{lem 312 n-1}, \ref{lem pt3}, and \ref{lem pt4}, $a_n(k)$ satisfies the above recurrence.
\end{proof}

 \subsection{Statistics and cycles in $\S_n(312, 4321)$}
 
In \cite{ELIZALDE2004}, Elizalde enumerated $\S_n(312, 4321)$ with respect to excedances and fixed points. In this section we note that the bijection $\varphi$ can be used to refine the generating function given in Theorem \ref{thm 312 4321} with respect to several statistics together with the number of cycles.

Let us first enumerate $\S_n(312, 4321)$ with respect to cycles, excedances, and inversions and in the second theorem, we enumerate  $\S_n(312, 4321)$ with respect to cycles and fixed points.  
 \begin{theorem}\label{312 4321 stats}
Let
\[
B(t,x,y,z) =  \sum_{n, \ell, m, j\geq 0} |\{\pi \in \S_n(312,4321) :  \cyc(\pi) = k,\exc(\pi) = \ell, \inv(\pi) = j\}| t^kx^\ell y^j z^n.
\]
 Then $$B(t,x,y,z) =  \frac{tz(1-x^2y^4z^2)}{1-z(xy+t)-xy^3z^2(t+xy)-x^2y^4z^3(y-t-xy)}.$$
 \end{theorem}
 
 \begin{proof}
 Let us suppose $n\geq 4$. It is enough to show $B$ satisfies the equation 
 $$ B =ztB+ zxyB + z^2txy^3B + z^3x^2y^5B + z^2x^2y^4B-z^3tB - z^3x^3y^5B+zt - tx^2y^4z^3.$$ 
As before, let us consider cases based on the location of $n$ in our permutation. 
 \begin{itemize}
 \item If $n$ appears in the last position, it is a fixed point. Deleting that fixed point results in a permutation with one fewer cycle and no additional excedances or inversions. This contributes the term $ztB$ to the functional equation above.
 \item If $n$ appears in the second-to-last position, then by deleting it, we end up with a permutation in 
 $\S_{n-1}(312, 4321)$ with the same number of cycles, one fewer excedance, and one fewer inversion. This contributes the term $zxyB$.

 \item If $n$ appears in the third-to-last position and $n-1$ appears in the second-to-last position, so that our permutation $\pi$ ends in $n(n-1)\pi_n$, then deleting both $n$ and $n-1$ leaves us with a 
 permutation in 
 $\S_{n-2}(312, 4321)$ with one fewer cycle (since $n-1$ was a fixed point), one fewer excedance, and three fewer inversions.  This contributes the term $z^2txy^3B$.

\item If $\pi \in \S_n(312, 4321)$ with $\pi_{n-2}=n$ and $\pi_{n-3}=n-1$, then we can delete $n$ and $n-1$ to get $\pi'$, as in the proof of Lemma~\ref{lem pt4}. How the number of excedances and inversions changes depends on the placement of $n-2$ in $\pi'$. 
 \begin{itemize}
 \item If $\pi'_{n-3}=n-2$, then $\pi'$ has no extra cycles, one fewer excedance, and four fewer inversions. 
Deleting $n-2$ from $\pi'$ to get $\pi''$, we would lose another excedance and inversion. Since $\pi''$ is now any permutation in $\S_{n-3}(312, 4321)$, we would contribute the term $z^3x^2y^5B$ to the functional equation. 
 \item If $\pi'_{n-3}\neq n-2$, then $\pi'$ has no extra cycles, two fewer excedances, and four fewer inversions. We do have to subtract off the cases where $n-2$ is a fixed point since, as seen in the proof of Lemma~\ref{lem pt4}, this cannot happen, and we should subtract off cases where $\pi'_{n-3}=n-2$ since that case was covered in the bullet point above. Therefore, we have the terms $z^2x^2y^4B-z^3tB - z^3x^3y^5B$.
 \end{itemize}
\end{itemize}
Taken together with the initial conditions (for $n<4$) we find that 
$$ B =zxyB + ztB+ z^2txy^3B + z^3x^2y^5B + z^2x^2y^4B-z^3tB - z^3x^3y^5B+zt - tx^2y^4z^3$$ which is equivalent to the statement of the theorem. 
\end{proof}

Next, we will refine the enumeration given by the generating function in Theorem \ref{thm 312 4321} with respect to fixed points.

\begin{theorem}\label{thm fp and cyc 312 4321}
Let
\[
C(t,u,z) =  \sum_{n, k,m\geq 0} |\{\pi \in \S_n(312,4321) : \fix(\pi) = m, \cyc(\pi) = k\}| t^ku^mz^n.
\]
 Then $$C(t,u,z) =  \frac{tuz+tz^2(1-u)+tuz^3(t-tu-1)+t^2z^5(1-u)^2}{1-z(1+tu)-z^2(1+t)+tuz^3(1+tu-t)+tz^4(u-1)-t^2z^5(1-u)^2}.$$
\end{theorem}

\begin{proof}
As before, let us consider cases based on the location of $n$ in our permutation. Let $n\geq 6$. 

  \begin{itemize}
 \item If $n$ appears in the last position, it is a fixed point. Deleting that fixed point results in a permutation with one fewer cycle and one fewer fixed point.  This contributes $ztuC$. 
 
 \item If $n$ appears in the second-to-last position, then by deleting it, we gain no cycles. If $n$ was in a 2-cycle with $n-1$, then we gain a fixed point and otherwise do not. Therefore, we obtain $z^2tC$ when $\pi_{n}=n-1$ and $zC-z^2tuC$ otherwise.   

 \item If $n$ appears in the third-to-last position and $n-1$ appears in the second-to-last position, so that our permutation $\pi$ ends in $n(n-1)\pi_n$, then deleting both $n$ and $n-1$ leaves us with a 
 permutation in 
 $\S_{n-2}(312, 4321)$ with one fewer cycle and one fewer fixed point (since $n-1$ was a fixed point). If additionally, $\pi_n=n-2$, then we gain a fixed point by deleting $n$ and $n-1$.
Therefore, we get $z^3t^2uC$ if $\pi_{n}=n-2$ and $(z^2tu-z^3t^2u)C$ otherwise. 
 
 \item If $\pi \in \S_n(312, 4321)$ with $\pi_{n-2}=n$ and $\pi_{n-3}=n-1$, then we can delete $n$ and $n-1$ to get $\pi'$, as in the proof of Lemma~\ref{lem pt4}. In this case, we add no cycles or fixed points. 
 For $\pi'\in\S'_{n-2}(312, 4321)$, $n-2$ is never a fixed point, but $n-3$ could be. If $n-3$ is a fixed point, then applying the proof of Lemma~\ref{lem pt3} to $\pi'$ to get $\pi''$, we lose that fixed point. Therefore we get $z^2u^{-1}(z^3t^2u+ z^2tu-z^3t^2u)C$.  If $n-3$ is not a fixed point of $\pi'$, then we get $z^2C$ but we must subtract the case when $n-2$ is fixed ($-z^3tuC$) and the case when $n-3$ is fixed, but $n-2$ is not ($z^2(z^3t^2u+ z^2tu-z^3t^2u)C$). 
\end{itemize}

Taken together with the initial conditions (when $n<6$), we get a functional equation equivalent to the statement of the theorem. 
\end{proof}
It is possible to combine Theorems \ref{312 4321 stats} and \ref{thm fp and cyc 312 4321} into one generating function using these techniques, but the answer is unwieldy. It is omitted here for that reason. 
%
%
%
%
%
%
%
%
 
 \subsection{Involutions in $\S_n(312, 4321)$}
 
 In this section, we recover and further refine a result in \cite{BBS2011} 
 that gives the number of involutions that avoid $312$ and $4321$ as the Tribonacci numbers. In that article, the authors use Motzkin paths to enumerate these involutions. Here we enumerate the involutions in $\S_n(312, 4321)$ with respect to fixed points, excedances, and number of cycles using the lemmas from the previous sections.
 
 
 
 \begin{theorem}\label{312 4321 involutions}
For all $n\geq 1$ and $k,\ell,m,j\geq0$, let $$d_n(m,\ell,k,j) = |\{\pi \in \I_n(312,4321) : \fix(\pi) = m, \exc(\pi) = \ell, \cyc(\pi) = k, \inv(\pi) = j\}|$$ and let
\[
D(t,u,x,y,z) =  \sum_{n,k, \ell, m, j\geq 0} d_n(k,\ell,m,j) t^ku^mx^\ell y^j z^n.
\]
 Then $$D(t,u,x,y,z) = \frac{tuz+txyz^2 + t^2uxy^3z^3}{1-utz-txyz^2-t^2uxy^3z^3}.$$
 \end{theorem}
 \begin{proof}
Let $n\geq 4$. First notice that for involutions, the number of excedances corresponds exactly to the number of 2-cycles. 
As before, let us consider the three cases based on where $n$ is in the permutation. 

If $n$ is at the end, then it is a fixed point, and removing it removes one cycle and does not change the number of excedances or inversions. This contributes $ztuD$ to the functional equation for $D$. 

If $n$ is in the second-to-last place, then it is an inversion only if $n-1$ appears in the last position. Then removing $n$ will leave us with an involution with the same number of cycles.  That is, the cycle $(n-1,n)$ would become the fixed point $(n-1)$. We will have one fewer excedance, and one fewer inversion. This contributes $z^2txyD$. 

If $n$ is in the third-to-last place, we must have that $(n-2,n)$ is a cycle in $\pi$. Removing $n$ and $n-1$ would leave us with a permutation $\pi'$ that has $n-2$ as a fixed point. By Lemma~\ref{lem where n-1}, we must have $n-1$ as a fixed point in $\pi$. Therefore, we should consider an involution $\pi''\in \mathcal{I}_{n-3}(312, 4321)$, then attach $n(n-1)(n-2)$ to the end of its one line notation (or equivalently, attach the cycles $(n-2,n)(n-1)$ to its cycle notation). This gives us the term $z^3t^2uxy^3D$. 

Together with the initial conditions, the result follows. 
 \end{proof}


\section{$\S_n(321, 4123)$}\label{321 4123}

In this section, we enumerate permutations that avoid $321$ and $4123$ with respect to several statistics. As in the previous section, in each case, we find a recurrence and obtain a generating function as a result.

\subsection{Number of cycles in $\S_n(321, 4123)$}

The set $\Av(321, 4123)$ was shown in \cite{West96}  to be enumerated by $F_{2n},$ the Fibonacci numbers with even index (just as $\Av(312, 4321)$ was). Our first result of this section is Theorem \ref{thm 321 4123} below, where we refine this enumeration with respect to cycles. 

 \begin{theorem}\label{thm 321 4123}
 Let $$F(t,z) = \sum_{n,k\geq 1} |\{\pi \in \S_n(321,4123) : \cyc(\pi) = k\}| t^kz^n.$$
 
 Then $$F(t,z) = \frac{tz(1-z^2)}{1-z(1+t)-2z^2+z^3}.$$
 \end{theorem}

The first few terms of this generating function are as follows:
\begin{align*}
F(t,z) &= tz + (t+t^2)z^2 + (2t+2t^2 + t^3)z^3 + (3t+6t^2+3t^3+t^4)z^4\\
&+ (6t+12t^2 + 11t^3 + 4t^4+ t^5) z^5 + (10t+28t^2 + 28t^3 + 17t^4 + 5t^5 + t^6)z^6 + \cdots.
\end{align*}

To prove Theorem~\ref{thm 321 4123}, let us first consider the following lemma. 

\begin{lemma}\label{lem 321 4123 where is n}
For $n\geq 1$, if $\pi \in \S_n(321, 4123)$ then the entry $n$ must appear in one of the last three entries of $\pi$.   
\end{lemma}
\begin{proof}
Suppose $\pi \in \S_n(312, 4321)$. Since $\pi$ avoids 321, any entries that appear after $n$ must be in increasing order, and since $\pi$ avoids 4123, there can be at most two such entries. 
\end{proof}

Since for any permutation in $\pi\in \S_n(321,4123)$, we need only consider the three cases where $\pi_n=n$, $\pi_{n-1}=n$, and $\pi_{n-2}=n$. As in the last section, we will consider these three cases separately. The first two cases are very similar to those cases from the last section.
Let $f_n(k)$ denote the number of permutations in $\S_n(321, 4123)$ that are composed of $k$ cycles. In the language of Theorem~\ref{thm 321 4123}, we have $F(t,z) = \sum f_n(k) t^kz^n$. 

\begin{lemma}\label{lem 321 n}
Let $n\geq 1$. The number of permutations $\pi \in \S_n(321, 4123)$ composed of $k$ cycles and with $\pi_n=n$ is given by $f_{n-1}(k-1)$. 
\end{lemma}
\begin{proof}
Removing the term $\pi_n=n$ from $\pi$, we are left with a permutation in $\S_{n-1}(321, 4123)$. In this case, since $n$ was a fixed point of $\pi$, by removing it, we have removed one cycle. This process is invertible since adding an $n$ at the end of a permutation cannot introduce a 321 or 4123 pattern.
\end{proof}

\begin{lemma}\label{lem 321 n-1}
Let $n\geq 1$. The number of permutations $\pi \in \S_n(321, 4123)$ composed of $k$ cycles and with $\pi_{n-1}=n$ is given by $f_{n-1}(k)$. 
\end{lemma}
\begin{proof}
Suppose $\pi\in\S_{n}(321, 4123)$ with $\pi_{n-1}=n$. Notice that $n-1, n,$ and $\pi_n$ appear consecutively in a cycle together.
Consider the permutation $\pi'$ obtained by removing the value $n$. In this case, $\pi'\in\S_{n-1}(321,4123)$ and we now have that $\pi'_{n-1}=\pi_n$, so $n-1$ maps directly to $\pi_n$ and no new cycle has been deleted or added. Since this is invertible, we can obtain any $\pi'\in\S_{n-1}(321,4123)$ in this way.
\end{proof}

Now, let us consider the case when $n$ appears in the $(n-2)$nd position.

\begin{lemma}\label{lem 321 where n-1}
For $n\geq 1$, if $\pi \in \S_n(321, 4123)$ with $\pi_{n-2}=n$, then either $\pi_{n}=n-1$ or $\pi_{n-3}=n-1$.
\end{lemma}
\begin{proof}
Suppose $\pi_{n}\neq n-1$. We must have that $\pi_{n-1}<\pi_n$ since $\pi$ avoids 321, so we cannot have $\pi_{n-1}=n-1$. If $\pi_i=n-1$ for $i<n-3$, then either $\pi_i \pi_{n-3}\pi_{n-1}$ is an occurrence of 321 (when $\pi_{n-3}>\pi_{n-1}$) or $\pi_i\pi_{n-3}\pi_{n-1}\pi_n$ is an occurrence of 4123 (when $\pi_{n-3}<\pi_{n-1}$). Therefore, if $\pi_{n}\neq n-1$, we have to have $\pi_{n-3}=n-1$. 
\end{proof}

\begin{lemma}\label{lem 321 pt3}
Let $n\geq1$. The number of permutations $\pi \in \S_n(321, 4123)$ composed of $k$ cycles and with $\pi_{n-2}=n$ and $\pi_{n}=n-1$ is $f_{n-2}(k)$. 
\end{lemma}
\begin{proof}
Suppose $\pi \in \S_n(321, 4123)$ with $\pi_{n-2}=n$ and $\pi_{n}=n-1$. Then $n-2, n,n-1,$ and $\pi_{n-1}$ appear consecutively in a cycle together. Let $\pi'$ be the permutation obtained by deleting $n$ and $n-1$. Then, $\pi' \in \S_{n-2}(321, 4123)$ with $\pi'_{n-2}=\pi_{n-1}$. In cycle notation, we have taken $\pi$, removed $n$ and $n-1$ from its cycle so that $n-2$ maps directly to $\pi_n$. Therefore, we have have not changed the number of cycles. Starting with any permutation $\pi' \in \S_{n-2}(321, 4123)$, we can invert this since $n\pi'_{n-2}(n-1)$ cannot be a 321 pattern.
\end{proof}

For example if $\pi = 245167938 = (124)(356798)$, then removing $n$ and $n-1$ would give you $\pi' = 2451673 = (124)(3567)$. There is no change in the number of cycles. 

\begin{lemma}\label{lem 321 pt4}
Let $n\geq1$. The number of permutations $\pi \in \S_n(321, 4123)$ composed of $k$ cycles and with $\pi_{n-2}=n$ and $\pi_{n-3}=n-1$ is $f_{n-2}(k)-f_{n-3}(k)$. 
\end{lemma}
\begin{proof}
For a permutation $\pi \in \S_n(321, 4123)$ with $\pi_{n-2}=n$ and $\pi_{n-3}=n-1$ we must have that $n-2, n,$ and $\pi_n$  appear consecutively in a cycle together and $n-3, n-1,$ and $\pi_{n-1}$ appear consecutively in a cycle together. Removing $n$ and $n-1$ to get a permutation $\pi'\in\S_{n-2}(321, 4123)$ with the same number of cycles. 

However, we can only obtain permutations $\pi'\in\S_{n-2}(321, 4123)$ so that $\pi'_{n-3}<\pi'_{n-2}$ since $\pi$ avoided 321. Since $\pi'$ itself avoids 321 and 4123, it is the case that $\pi'_{n-3}<\pi'_{n-2}$ exactly when $\pi'_{n-3}\neq n-2$. Therefore, we obtain all permutations in $\S_{n-2}(321, 4123)$ with the same number of cycles except those where $\pi'_{n-3}=n-2$

If a permutation in  $\S_{n-2}(321, 4123)$ has $k$ cycles and $\pi'_{n-3}=n-2$, we can remove $n-2$ and get a permutation in $\S_{n-3}(321, 4123)$ with the same number of cycles. Through this process we can get any permutation in $\S_{n-3}(321, 4123)$, so the result follows. 
\end{proof}

As an example, consider $\pi = 213478956=(12)(3)(4)(57968)$. Removing $n=9$ and $n-1=8$, we obtain $\pi' = 2134756=(12)(3)(4)(576)$
which has the same number of cycles and $\pi'_{6}<\pi'_7$.

\begin{proof}[Proof of Theorem \ref{thm 321 4123}]

 Taken together with the initial conditions, it is enough for us to show that for $n\geq 4$, the following recurrence holds:
$$ f_n(k) = f_{n-1}(k) + f_{n-1}(k-1) + 2f_{n-2}(k) -f_{n-3}(k).$$
By Lemmas~\ref{lem 321 n}, \ref{lem 321 n-1}, \ref{lem 321 pt3}, and \ref{lem 321 pt4}, $f_n(k)$ satisfies the above recurrence.
\end{proof}

The following corollary regarding the number of \textit{cyclic} permutations that avoid 321 and 4123 follows immediately from this proof. 

  \begin{corollary}
  The number of \textit{cyclic} permutations $\bar{f}_n:=f_n(1)$ avoiding the patterns 321 and 4123 satisfies the recurrence: 
  $$
  \bar{f}_n = \bar{f}_{n-1} + 2\bar{f}_{n-2} - \bar{f}_{n-3}
  $$
  and is given by the OEIS sequence A028495.
  \end{corollary}

\subsection{Statistics and cycles in $\S_n(321, 4123)$}

Next, we refine the enumeration from the previous section with respect to several variables, including excedances, fixed points, and inversions.
 
 
  \begin{theorem}\label{321 4123 statistics }
For all $n\geq 1$ and $k,\ell,m,j\geq0$, let $$g_n(k,\ell,m,j) = |\{\pi \in \S_n(321,4123) : \cyc(\pi) = k, \exc(\pi) = \ell, \fix(\pi) = m, \inv(\pi) = j\}|$$ and let
\[
G(t,u,x,y,z) =  \sum_{n,k, \ell, m, j\geq 0} g_n(k,\ell,m,j) t^ku^mx^\ell y^jz^n.
\]
 Then 
 \begin{align*}
 &G(t,u,x,y,z) = \frac{tz(u+xyz(1-u)+xy^2z^2(1-uxy^2-u)+x^2y^4z^3(xy-t)(u-1))}{1-z(xy+tu) + z^2\mathfrak{g}_2(t,u,x,y) + z^3\mathfrak{g}_3(t,u,x,y) + z^4\mathfrak{g}_4(t,u,x,y)}
 \end{align*}
 where 
 \vspace{-12pt}
  \begin{align*}
  &\mathfrak{g}_2(t,u,x,y)= xy(t(u-1)-y-xy^3), \\ &\mathfrak{g}_3(t,u,x,y) = xy^2(ut-t+xy^2(ut+xy-t)), \text{ and}\\ &\mathfrak{g}_4(t,u,x,y) = tx^2y^4(xy-t)(1-u). \end{align*}
 \end{theorem}
 
%
\begin{proof}
Let $n\geq 5$. As before, let us consider cases based on the location of $n$ in our permutation.

\begin{itemize}
\item If $n$ appears in the last position, it is a fixed point. Deleting that fixed point results in a permutation with one fewer cycle, one few fixed point, and no additional excedances or inversions. This contributes the term $ztuG$ to the functional equation above.
\item If $n$ appears in the second-to-last position, then by deleting it, we end up with a permutation in 
 $\S_{n-1}(321, 4123)$ with no extra cycles, one fewer excedance, and one fewer inversion. If $n$ is in a 2-cycle with $n-1$, then we get the term $z^2txyG$ by starting with a permutation of length $n-2$ with one fewer cycle and then adding on the cycle $(n-1,n)$. If $n$ is not in a 2-cycle with $n-1$, deleting $n$ will not change the number of fixed points, and will leave $n-2$ in a cycle of size greater than one, so we get the term $(zxy-z^2tuxy)G$.
\item If $n$ appears in the third-to-last position, and $n-1$ appears in the last position, then $n-2, n,n-1,$ and $\pi_{n-1}$ appear consecutively in a cycle together. Deleting $n$ and $n-1$ leaves you with the same number of cycles, two fewer inversions, and one fewer excedance.  If $\pi_{n-1}=n-2$, then doing this gains you a fixed point and otherwise the number of fixed points is unchanged. 
Therefore we get $z^3txy^2G$ if $\pi_{n-1}=n-2$ since we can start with a permutation of length $n-3$ and add on the cycle $(n-2, n, n-1)$. If $\pi_{n-1}\neq n-2$, we will not get an extra fixed point by deleting $n$ and $n-1$, so we get $z^2xy^2G-z^3tuxy^2G$.
\item If $\pi \in \S_n(321, 4123)$ with $\pi_{n-2}=n$ and $\pi_{n-3}=n-1$, then we can delete $n$ and $n-1$ to get $\pi'$, as in the proof of Lemma~\ref{lem 321 pt4}.
We add four inversions and two excedances. If $n-2$ and $n-3$ are both fixed points of $\pi$, then we obtain $z^4t^2x^2y^4G$. 
If $n-2$ is fixed, but $n-3$ is not, then we obtain $z^3tx^2y^4G -z^4t^2ux^2y^4G$. If $n-2$ is not fixed and $\pi_{n-3}\neq n-2$, then we must have that $\pi_{n-3}<\pi_{n-2}$, and so $\pi_{n-3}$ is not fixed either. Therefore we obtain $z^2x^2y^4G$. We must subtract off the case when $n-2$ is fixed (which gives us $-z^3tux^2y^4$) and the case that $\pi_{n-3}=n-2$, since we do not allow this. For the latter possibility, there are two subcases: when the permutation has a two cycle $(n-3,n-2)$ or when it does not. In the first subcase, we subtract $z^4tx^3y^5G$ and in the second subcase, we subtract $(z^3x^3y^5 +z^4tux^3y^5)G$.
\end{itemize}

Together with the initial conditions, this gives us a functional equation
which is equivalent to the statement of the theorem. 
 \end{proof}
 
\subsection{Involutions in $\S_n(321, 4123)$ }

 In this section, we enumerate the involutions in $\S_n(321, 4123)$ with respect to fixed points, excedances, number of cycles, and inversions. 
 
 
 
 \begin{theorem}\label{321 4123 involutions}
For all $n\geq 1$ and $k,\ell,m,j\geq0$, let $$h_n(k,\ell,m,j) = |\{\pi \in \I_n(321,4123) : \fix(\pi) = m, \exc(\pi) = \ell, \cyc(\pi) = k, \inv(\pi) = j\}|$$ and let
\[
H(x, t,w,y,z) =  \sum_{n,k, \ell, m, j\geq 0} h_n(k,\ell,m,j)  t^ku^mx^\ell y^jz^n.
\]
 Then $$H(x,t,u,v,z) = \frac{tuz+txyz^2+t^2x^2y^4z^4}{1- tuz -txyz^2- t^2x^2y^4z^4}.$$
 \end{theorem}
 \begin{proof}
As before, notice that for involutions, the number of excedances corresponds exactly to the number of 2-cycles. 
Again, let us consider the position of $n$ in the permutation $\pi \in \I_n(321, 4123)$. If $n$ is at the end, it is a fixed point, and deleting it will give you an involution that has one fewer fixed point and one fewer cycle. No inversions or excedances are lost. This gives you the term $ztuH$.

If $n$ is in the second-to-last position, then it must be the case that $\pi_n={n-1}$. This contributes $z^2txyH$ since you can start with any involution in $\I_{n-2}(321, 4123)$ and append the 2-cycle $(n-1,n)$ to the cycle notation. The adds one cycle, no fixed points, one excedance, and one inversion. 

Finally, if $n$ is in the third-to-last position, we must have that $\pi_{n-3}\pi_{n-2}\pi_{n-1}\pi_{n} = (n-1)n(n-3)(n-2)$ in order for the permutation to be an involution.
Therefore, this part contributes $z^4t^2x^2y^4H$ to the recurrence since we can start with a permutation $\I_{n-4}(321,4123)$, attach the two 2-cycles $(n-3,n-1)$ and $(n-2,n)$ to get $\pi'$. This adds two cycles, no fixed points, two excedances, and four inversions. 
\end{proof}

\bibliographystyle{plain}

\end{document}